\documentclass[12pt]{amsart}
\usepackage{graphicx}
\usepackage{amsfonts, amssymb, amsmath, amsthm, amscd}
\usepackage[dvips]{epsfig}
\usepackage{bm}
\usepackage{url}
\usepackage[all]{xy}
\usepackage[dvips]{epsfig}
\usepackage[unicode, bookmarksnumbered=true, backref=true, dvips]{hyperref}
\usepackage{color}

\def\R{\mathbb{R}}

\def\Z{\mathbb{Z}}

\def\T{\mathbb{T}}
\def\AA{\mathcal{A}}
\def\BB{\mathcal{B}}
\def\GG{\mathcal{G}}
\def\HH{\mathcal{H}}

\def\SS{\mathfrak{S}}

\def\x{\mathbf{x}}

\newcommand{\CF} {\widehat{\operatorname{CF}}}
\newcommand{\HF} {\widehat{\operatorname{HF}}}

\newcommand{\abs}[1] {\lvert #1 \rvert}
\newcommand{\gen}[1] {\left\langle #1 \right\rangle}

\def\minus {\smallsetminus}
\def\Th {{}^{\text{th}}}

\newtheorem{theorem}{Theorem}
\newtheorem{lemma}{Lemma}

\newtheorem{conj}{Conjecture}

 \DeclareMathOperator{\sign}{sign}
\DeclareMathOperator{\Sym}{Sym} 
\DeclareMathOperator{\gr}{gr} \DeclareMathOperator{\rank}{rank}

\title{Strong L-spaces and left-orderability}

\author{Adam Simon Levine}
\address{Department of Mathematics, Brandeis University, Waltham, MA 02453}
\email{levinea@brandeis.edu}

\author{Sam Lewallen}
\address{Department of Mathematics,
Princeton University, Princeton, NJ 08544}
\email{lewallen@math.princeton.edu}

\thanks{The first author was supported by an NSF Postdoctoral Fellowship.}

\begin{document}

\begin{abstract}
We introduce the notion of a strong L-space, a closed, oriented rational
homology 3-sphere whose Heegaard Floer homology can be determined at the chain
level. We prove that the fundamental group of a strong L-space is not
left-orderable. Examples of strong L-spaces include the double branched covers
of alternating links in $S^3$.
\end{abstract}

\maketitle

\section{Introduction}

Heegaard Floer homology, developed by Ozsv\'ath and Szab\'o \cite{OSz3Manifold}
in the early 2000s, has been an extremely effective tool for answering
classical questions about 3-manifolds, particularly concerning the genera of
embedded surfaces \cite{OSzGenus}. However, surprisingly little is known about
the relationship between Heegaard Floer homology and topological properties of
Heegaard splittings, even though a Heegaard diagram is an essential ingredient
in defining the Heegaard Floer homology of a closed $3$-manifold $Y$. In
particular, a Heegaard diagram provides a presentation of the fundamental group
of $Y$, and it is natural to ask how this presentation is related to the
Heegaard Floer chain complex. In this paper, we shall investigate one such
connection.

A \emph{left-ordering} on a non-trivial group $G$ is a total order $<$ on the
elements of $G$ such that $g < h$ implies $kg < kh$ for any $g,h,k \in G$. A
group $G$ is called \emph{left-orderable} if it is nontrivial and admits at
least one left-ordering. The question of which $3$-manifolds have
left-orderable fundamental group has been of considerable interest and is
closely connected to the study of foliations. For instance, if $Y$ admits an
$\R$-covered foliation (i.e., a taut foliation such that the leaf-space of the
induced foliation on the universal cover $\widetilde{Y}$ is homeomorphic to
$\R$), then $\pi_1(Y)$ is left-orderable. Boyer, Rolfsen, and Wiest
\cite{BoyerRolfsenWiest} showed that the fundamental group of any irreducible
$3$-manifold $Y$ with $b_1(Y)>0$ is left-orderable, reducing the question to
that of rational homology spheres.

In its simplest form, Heegaard Floer homology associates to a closed, oriented
$3$-manifold $Y$ a $\Z/2\Z$--graded, finitely generated abelian group $\HF(Y)$.
This group is computed as the homology of a free chain complex $\CF(\HH)$
associated to a Heegaard diagram $\HH$ for $Y$; different choices of diagrams
for the same manifold yield chain-homotopy-equivalent complexes. The group
$\CF(\HH)$ depends only on the combinatorics of $\HH$, but the differential on
$\CF(\HH)$ involves counts of holomorphic curves that rely on auxiliary choices
of analytic data. If $Y$ is a rational homology sphere, then the Euler
characteristic of $\HF(Y)$ is equal to $\abs{H_1(Y;\Z)}$, which implies that
the rank of $\HF(Y)$ is at least $\abs{H_1(Y;\Z)}$. $Y$ is called an
\emph{L-space} if $\HF(Y) \cong \Z^{\abs{H_1(Y;\Z)}}$; thus, L-spaces have the
simplest possible Heegaard Floer homology. Examples of L-spaces include $S^3$,
lens spaces (whence the name), all manifolds with finite fundamental group, and
double branched covers of alternating (or, more broadly,
\emph{quasi-alternating}) links. Additionally, Ozsv\'ath and Szab\'o
\cite{OSzGenus} showed that if $Y$ is an L-space, it does not admit any taut
foliation; whether the converse is true is an open question.

The following related conjecture, stated formally by Boyer, Gordon, and Watson
\cite{BoyerGordonWatson}, has recently been the subject of considerable
attention:
\begin{conj} \label{conj:Lspace}
Let $Y$ be a closed, connected, 3-manifold. Then $\pi_{1}(Y)$ is not
left-orderable if and only if $Y$ is an L-space.
\end{conj}
This conjecture is now known to hold for all geometric, non-hyperbolic
3-manifolds \cite{BoyerGordonWatson}.\footnote{Specifically, work of Boyer,
Rolfsen, and Wiest \cite{BoyerRolfsenWiest} and Lisca and Stipsicz
\cite{LiscaStipsiczInvariants3} gives the result for Seifert manifolds with
base orbifold $S^2$, as was also observed by Peters \cite{PetersLSpaces}. The
cases of Seifert manifolds with non-orientable base orbifold and of Sol
manifolds follow from \cite{BoyerRolfsenWiest} and \cite{BoyerGordonWatson}.}
Additionally, Boyer, Gordon, Watson \cite{BoyerGordonWatson} and Greene
\cite{GreeneAlternating} have shown that the double branched cover of any
non-split alternating link in $S^3$
--- which is generically a hyperbolic $3$-manifold --- has non-left-orderable
fundamental group.

In this paper, we prove the ``if'' direction of Conjecture \ref{conj:Lspace}
for manifolds that are ``L-spaces on the chain level.'' To be precise, we call
a 3-manifold $Y$ a \emph{strong L-space} if it admits a Heegaard diagram $\HH$
such that $\CF(\HH) \cong \Z^{\abs{H_1(Y;\Z)}}$. This purely combinatorial
condition implies that the differential on $\CF(\HH)$ vanishes, without any
consideration of holomorphic disks. We call such a Heegaard diagram a
\emph{strong Heegaard diagram}. By considering the presentation for $\pi_1(Y)$
associated to a strong Heegaard diagram, we prove:
\begin{theorem} \label{thm:main}
If $Y$ is a strong L-space, then $\pi_1(Y)$ is not left-orderable.
\end{theorem}

The standard Heegaard diagram for a lens space is easily seen to be a strong
diagram. Moreover, Greene \cite{GreeneSpanning} constructed a strong Heegaard
diagram for the double branched cover of any alternating link in $S^3$; indeed,
Boyer, Gordon, and Watson's proof that the fundamental group of such a manifold
is not left-orderable essentially makes use of the group presentation for
$\pi_1$ associated to that Heegaard diagram. At present, we do not know of any
strong L-space that cannot be realized as the double branched cover of an
alternating link; while it seems unlikely that every strong L-space can be
realized in this manner, it is unclear what obstructions could be used to prove
this claim. (Indeed, the question of finding an alternate characterization of
alternating links is a famous open problem posed by R. H. Fox.) Nevertheless,
our theorem seems like a useful step in the direction of Conjecture
\ref{conj:Lspace} in that it relies only on data contained in the Heegaard
Floer chain complex.

On the other hand, the following theorem, which is well-known but does not
appear in the literature, does indicate that being a strong L-space may be a
fairly restrictive condition:
\begin{theorem} \label{thm:S3}
If $Y$ is an integer homology sphere that is a strong L-space, then $Y \cong
S^3$.
\end{theorem}
In particular, there exist integer homology spheres that are L-spaces (e.g.,
the Poincar\'e homology sphere) but not strong L-spaces. The fact that the
condition of being a strong L-space detects $S^3$ suggests that it might be
possible to obtain a more explicit characterization or even a complete
classification of strong L-spaces. Below, we shall present a graph-theoretic
proof of Theorem \ref{thm:S3} due to Josh Greene. In fact, this proof can be
extended to classify the finitely many strong L-spaces with $|H_1(Y;\Z)|\leq
3$, and it is natural to ask whether, for any $n$, there are finitely many
strong L-spaces with $|H_1(Y;\Z)| \leq n$.

\subsection*{Acknowledgments} The authors are grateful to Josh
Greene, Eli Grigsby, Peter Ozsv\'ath, and Liam Watson for helpful
conversations, and to the Simons Center for Geometry and Physics, where much of
the work in this paper was completed while the authors were visiting in May
2011.

\section{Proofs of Theorem \ref{thm:main} and \ref{thm:S3}}

To prove Theorem \ref{thm:main}, we will use a simple obstruction to
left-orderability that can be applied to group presentations.

Let $X$ denote the set of symbols $\{0,+,-,*\}$. These symbols are meant to
represent the possible signs of real numbers: $+$ and $-$ represent positive
and negative numbers, respectively, and $*$ represents a number whose sign is
not known. As such, we define a commutative, associative multiplication
operation on $X$ by the following rules: (1) $0 \cdot \epsilon = \epsilon \cdot
0 = 0$ for any $\epsilon \in X$; (2) $+ \cdot + = - \cdot - = +$; (3) $+ \cdot
- = - \cdot + = -$; and (4) $\epsilon \cdot
* =
* \cdot \epsilon =
*$ for $\epsilon \in \{+,-,*\}$.

A group presentation $\GG = \gen{ x_1,\dots, x_m | r_1, \dots, r_n}$ gives rise
to an $m \times n$ matrix $E(\GG) = (\epsilon_{i,j})$ with entries in $X$ by
the following rule:
\begin{equation} \label{eq:epsilonij}
\epsilon_{i,j} = \begin{cases}
 0 & \text{if neither $x_i$ nor $x_i^{-1}$ occur in $r_j$} \\
 + & \text{if $x_i$ appears in $r_j$ but $x_i^{-1}$ does not} \\
 - & \text{if $x_i^{-1}$ appears in $r_j$ but $x_i$ does not} \\
 * & \text{if both $x_i$ and $x_i^{-1}$ occur in $r_j$}. \\
\end{cases}
\end{equation}

\begin{lemma} \label{lemma:notLO}
Let $\GG = \gen{ x_1,\dots, x_m | r_1, \dots, r_n}$ be a group presentation
such that for any $d_1, \dots, d_m \in \{0,+,-\}$, not all zero, the matrix $M$
obtained from $E(\GG)$ by multiplying the $i\Th$ row by $d_i$ has a nonzero
column whose nonzero entries are either all $+$ or all $-$. Then the group $G$
presented by $\GG$ is not left-orderable.
\end{lemma}

\begin{proof}
Suppose that $<$ is a left-ordering on $G$, and let $d_i$ be $0$, $+$, or $-$
according to whether $x_i=1$, $x_i>1$, or $x_i<1$ in $G$. Since $G$ is
nontrivial, at least one of the $d_i$ is nonzero. If the $j\Th$ column of $M$
is nonzero and has entries in $\{0,+\}$, the relator $r_j$ is a product of
generators $x_i$ that are all nonnegative in $G$, and at least one of which is
strictly positive. Thus, $r_j>1$ in $G$, which contradicts the fact that $r_j$
is a relator. An analogous argument applies for a nonzero column with entries
in $\{0,-\}$.
\end{proof}

We shall focus on presentations with the same number of generators as
relations. For a permutation $\sigma \in S_n$, let $\sign(\sigma) \in \{+,-\}$
denote the sign of $\sigma$ ($+$ if $\sigma$ is even, $-$ if $\sigma$ is odd).
The key technical lemma is the following:

\begin{lemma} \label{lemma:matrix}
Let $\GG = \gen{ x_1,\dots, x_n | r_1, \dots, r_n}$ be a group presentation
such that $E(\GG)$ has the following properties:
\begin{enumerate}
\item There exists at least one permutation $\sigma_0 \in S_n$ such that the
entries $\epsilon_{1,\sigma_0(1)}, \dots, \epsilon_{n,\sigma_0(n)}$ are all
nonzero.
\item For any permutation $\sigma \in S_n$ such that
$\epsilon_{1,\sigma(1)}, \dots, \epsilon_{n,\sigma(n)}$ are all nonzero, we
have $\epsilon_{1,\sigma(1)}, \dots, \epsilon_{n,\sigma(n)} \in \{+,-\}$.
\item For any two permutations $\sigma, \sigma'$ as in (2), we have
\[
\sign(\sigma) \cdot \epsilon_{1,\sigma(1)} \cdot \dots \cdot
\epsilon_{n,\sigma(n)} = \sign(\sigma') \cdot \epsilon_{1,\sigma'(1)} \cdot
\dots \cdot \epsilon_{n,\sigma'(n)}.
\]
\end{enumerate}
Then the group $G$ presented by $\GG$ is not left-orderable.
\end{lemma}

In other words, if we consider the formal determinant
\[
\det(E(\GG)) = \sum_{\sigma \in S_n} \sign(\sigma) \cdot \epsilon_{1,\sigma(1)}
\cdot \dots \cdot \epsilon_{n,\sigma(n)},
\]
condition (1) says that at least one summand is nonzero, condition (2) says
that no nonzero summand contains a $*$, and condition (3) says that every
nonzero summand has the same sign.

\begin{proof}
By reordering the generators and relations, it suffices to assume that
$\sigma_0$ from condition (1) is the identity, so that $\epsilon_{i,i} \ne 0$
for $i= 1, \dots, n$, and hence $\epsilon_{i,i} \in \{+,-\}$ by condition (2).
We shall show that $E(\GG)$ satisfies the hypotheses of Lemma
\ref{lemma:notLO}.

Suppose, then, toward a contradiction, that $d_1, \dots, d_n$ are elements of
$\{0,+,-\}$, not all zero, such that every nonzero column of the matrix $M$
obtained as in Lemma \ref{lemma:notLO} contains a nonzero off-diagonal entry
(perhaps a $*$) that is not equal to the diagonal entry in that column. Denote
the $(i,j)\Th$ entry of $M$ by $m_{i,j}$.

\begin{figure}
\[
\begin{pmatrix}
 + & 0 & 0 & 0 & \fbox{$-$} \\
 0 & + & \fbox{$-$} & * & 0 \\
 \fbox{$-$} & 0 & + & * & 0 \\
 0 & 0 & 0 & \fbox{$+$} & 0 \\
 + & \fbox{$-$} & 0 & 0 & + \\
\end{pmatrix}
\qquad \qquad
 \left(
 \vcenter{ \xymatrix@R=6pt@C=6pt{
 + \ar[dd]  &               &           &       & - \ar[llll] \\
            & + \ar[ddd]    & - \ar[l]  & {*}   &   \\
 - \ar[rr]  &               & + \ar[u]  & {*}   &   \\
            &               &           & +     &   \\
 +          & - \ar[rrr]    &           &       & + \ar[uuuu]
}} \right)
\]
\caption{Illustration of the proof of Lemma \ref{lemma:matrix}. In the matrix
$M$ shown at left, the entries $m_{i, \sigma(i)}$ are highlighted, where
$\sigma$ is the permutation constructed in the proof. To find $\sigma$, we
start with the $+$ in the upper left corner, travel to a $-$ in the same
column, and then travel to the diagonal entry in the same row as this $-$.
Repeating this procedure, we eventually obtain a closed loop, as shown at
right.} \label{fig:connect}
\end{figure}

We may inductively construct a sequence of distinct indices $i_1, \dots, i_k
\in \{1, \dots, n\}$ such that
\begin{enumerate}
\item[(A)] $m_{i_j,i_j} \in \{+,-\}$ for each $j=1, \dots, m$,
and
\item[(B)]$m_{i_{j+1},i_{j}} \ne 0$ and $m_{i_{j+1},i_j} \ne
m_{i_j,i_j}$
\end{enumerate}
for each $j=1, \dots, k$, taken modulo $k$. This is done by ``connecting the
dots'' as in Figure \ref{fig:connect}. Specifically, we begin by choosing any
$i_1$ such that $m_{i_1,i_1} \ne 0$. Given $i_j$, our assumption on $M$ states
that we can choose $i_{j+1}$ satisfying assumption (B) above; we then have
$m_{i_{j+1},i_{j+1}} \ne 0$ since otherwise the whole $i_{j+1}\Th$ row would
have to be zero. Repeating this procedure, we eventually obtain an index $i_k$
that is equal to some previously occurring index $i_{k'}$, where $k'+1 <k$. The
sequence $i_{k'+1}, \dots, i_k$, relabeled accordingly, then satisfies the
assumptions (A) and (B).

Define a $k$-cycle $\sigma\in S_n$ by $\sigma(i_j) = i_{j+1}$ for $j=1, \dots,
k$ mod $k$, and $\sigma(i') = i'$ for $i' \not\in \{i_1, \dots, i_k\}$. By
construction, $\epsilon_{i, \sigma(i)} \ne 0$ for each $i = 1, \dots, n$, so
the sequence $(\epsilon_{1, \sigma(1)}, \dots, \epsilon_{n,\sigma(n)})$
contains no $*$s by condition (2). The sequences $(\epsilon_{1, \sigma(1)},
\dots, \epsilon_{n,\sigma(n)})$ and $(\epsilon_{1,1}, \dots, \epsilon_{n,n})$
differ in exactly $k$ entries, and the signature of $\sigma$ is $(-1)^{k-1}$.
This implies that
\[
\sign(\sigma) \cdot \epsilon_{1, \sigma(1)} \cdot \dots \cdot
\epsilon_{n,\sigma(n)} = (-1)^{2k-1} \sign(\operatorname{id}) \cdot
\epsilon_{1, 1} \cdot \dots \cdot \epsilon_{n,n},
\]
which contradicts condition (3). This completes the proof.
\end{proof}

Now we will apply Lemma \ref{lemma:matrix} to prove Theorem \ref{thm:main}. We
first recall some basic facts about the Heegaard Floer chain complex. A
\emph{Heegaard diagram} is a tuple $\HH = (\Sigma, \bm\alpha, \bm\beta)$, where
$\Sigma$ is a closed, oriented surface of genus $g$, $\bm\alpha = (\alpha_1,
\dots, \alpha_g)$ and $\bm\beta = (\beta_1, \dots, \beta_g)$ are each
$g$-tuples of pairwise disjoint simple closed curves on $\Sigma$ that are
linearly independent in $H_1(\Sigma;\Z)$, and each pair of curves $\alpha_i$
and $\beta_j$ intersect transversely.  A Heegaard diagram $\HH$ determines a
closed, oriented $3$-manifold $Y = Y_\HH$ with a self-indexing Morse function
$f: Y \to [0,3]$ such that $\Sigma = f^{-1}(3/2)$, the $\alpha$ circles are the
belt circles of the $1$-handles of $Y$, and the $\beta$ circles are the
attaching circles of the $2$-handles. If we orient the $\alpha$ and $\beta$
circles, the Heegaard diagram determines a group presentation
\[
\pi_1(Y) = \gen{a_1, \dots, a_g \mid b_1, \dots, b_g},
\]
where the generators $a_1, \dots, a_g$ correspond to the $\alpha$ circles, and
$b_j$ is the word obtained as follows: If $p_1, \dots, p_k$ are the
intersection points of $\beta_j$ with the $\alpha$ curves, indexed according to
the order in which they occur as one traverses $\beta_i$, and $p_\ell \in
\alpha_{i_\ell} \cap \beta_i$ for $\ell = 1, \dots, k$, then
\begin{equation} \label{eq:relation}
b_j = \prod_{\ell = 1}^k a_{i_\ell}^{\eta(p_i)},
\end{equation}
where $\eta(p_i) \in \{\pm 1\}$ is the local intersection number of
$\alpha_{i_\ell}$ and $\beta_j$ at $p_i$.

Let $\Sym^g(\Sigma)$ denote the $g\Th$ symmetric product of $\Sigma$, and let
$\T_\alpha, \T_\beta \subset \Sym^g(\Sigma)$ be the $g$-dimensional tori
$\alpha_1 \times \dots \times \alpha_g$ and $\beta_1 \times \dots \times
\beta_g$, which intersect transversely in a finite number of points. Assuming
$Y$ is a rational homology sphere, $\CF(\HH)$ is the free abelian group
generated by points in $\SS_\HH = \T_\alpha \cap \T_\beta$.\footnote{For
general $3$-manifolds, we must restrict to a particular class of so-called
admissible diagrams.} More explicitly, these are tuples $\x = (x_1, \dots,
x_g)$, where $x_i \in \alpha_i \cap \beta_{\sigma(i)}$ for some permutation
$\sigma \in S_g$. The differential on $\CF(\HH)$ counts holomorphic Whitney
disks connecting points of $\SS_\HH$ (and depends on an additional choice of a
basepoint $z \in \Sigma$), but we do not need to describe this in any detail
here.

Orienting the $\alpha$ and $\beta$ circles determines orientations of
$\T_\alpha$ and $\T_\beta$. For $\x \in \SS_\HH$, let $\eta(\x)$ denote the
local intersection number of $\T_\alpha$ and $\T_\beta$ at $\x$. If $\x = (x_1,
\dots, x_g)$ with $x_i \in \alpha_i \cap \beta_{\sigma(i)}$, we have
\begin{equation} \label{eq:grading}
\eta(\x) = \sign(\sigma) \prod_{i=1}^g \eta(x_i).
\end{equation}
These orientations determine a $\Z/2$-valued grading $\gr$ on $\CF(Y)$ by the
rule that $(-1)^{\gr(\x)} = \eta(\x)$; the differential shifts this grading by
$1$. If $Y$ is a rational homology sphere, then with respect to this grading,
we have $\chi( \CF(\HH)) = \pm \abs{H_1(Y;\Z)}$, and we may choose the
orientations such that the sign is positive. (See \cite[Section
5]{OSzProperties} for further details.)

The proof of Theorem \ref{thm:main} is completed with the following:

\begin{lemma}
If $\HH$ is a strong Heegaard diagram for a strong L-space $Y$, then the
corresponding presentation for $\pi_1(Y)$ satisfies the hypotheses of Lemma
\ref{lemma:matrix}.
\end{lemma}

\begin{proof}
If $\rank( \CF(\HH)) = \chi(\CF(\HH)) = \abs{H_1(Y;\Z)}$, then $\CF(\HH)$ is
supported in a single grading, so $\eta(\x) = 1$ for all $\x \in \T_\alpha \cap
\T_\beta$. The result then follows quickly from equations \eqref{eq:epsilonij},
\eqref{eq:relation}, and \eqref{eq:grading}. Specifically, since $\SS_\HH \ne
\emptyset$, there exists $\sigma_0 \in S_g$ such that $\alpha_i \cap
\beta_{\sigma_0(i)} \ne \emptyset$ for each $i$, and hence $\epsilon_{i,
\sigma_0(i)} \ne 0$. If $\alpha_i$ and $\beta_j$ contain a point $x$ that is
part of some $\x \in \SS_\HH$, then every other point $x' \in \alpha_i \cap
\beta_j$ has $\eta(x') = \eta(x)$, and hence $\epsilon_{i,j} = \eta(\x) \in
\{+, -\}$. Finally, if $\x = (x_1, \dots, x_g)$ and $\x' = (x'_1, \dots,
x'_g)$, with $x_i \in \alpha_i \cap \beta_{\sigma(i)}$ and $x'_i \in \alpha_i
\cap \beta_{\sigma'(i)}$, then equation \eqref{eq:grading} and the fact that
$\eta(\x) = \eta(\x')$ imply the final hypothesis.
\end{proof}

Finally, to prove Theorem \ref{thm:S3}, we use a simple graph-theoretic
argument. Given a Heegaard diagram $\HH$, let $\Gamma_\HH$ denote the bipartite
graph with vertex sets $\AA = \{A_1, \dots, A_g\}$ and $\BB = \{B_1, \dots,
B_g\}$, with an edge connecting $A_i$ and $B_j$ for each intersection point in
$\alpha_i \cap \beta_j$. The set $\SS_\HH$ thus corresponds to the set of
perfect matchings on $\Gamma_\HH$.

\begin{lemma} \label{lemma:destab}
If $\HH$ is a Heegaard diagram of genus $g>1$, and $\Gamma_\HH$ contains a leaf
(a $1$-valent vertex), then $Y_\HH$ admits a Heegaard diagram $\HH'$ of genus
$g-1$ with a bijection between $\SS_\HH$ and $\SS_{\HH'}$.
\end{lemma}

\begin{proof}
If $A_i$ is $1$-valent, then the curve $\alpha_i$ intersects one $\beta$ curve,
say $\beta_j$, in a single point and is disjoint from the remaining $\beta$
curves. By a sequence of handleslides of the $\alpha$ curves, we may remove any
intersections of $\beta_j$ with any $\alpha$ curve other than $\alpha_i$,
without introducing or removing any intersection points. We may then
destabilize to obtain $\HH'$. Since every element of $\SS_\HH$ includes the
unique point of $\alpha_i \cap \beta_j$, we have a bijection between $\SS_\HH$
and $\SS_{\HH'}$. (Indeed, $\Gamma_\HH'$ is obtained from $\Gamma_\HH$ by
deleting $A_i$ and $B_j$, which does not change the number of perfect
matchings.) The case where $B_i$ is $1$-valent is analogous.
\end{proof}

\begin{proof}[Proof of Theorem \ref{thm:S3}]
Let $\HH$ be a strong Heegaard diagram for $Y$ whose genus $g$ is minimal among
all strong Heegaard diagrams for $Y$. Suppose, toward a contradiction, that
$g>1$. By Lemma \ref{lemma:destab}, $\Gamma_\HH$ has no leaves. By assumption,
$\Gamma_\HH$ has a single perfect matching $\mu$. We direct the edges of
$\Gamma_\HH$ by the following rule: an edge points from $\AA$ to $\BB$ if it is
included in $\mu$ and from $\BB$ to $\AA$ otherwise. Thus, every vertex in
$\AA$ has exactly one outgoing edge, and every vertex in $\BB$ has exactly one
incoming edge. We claim that $\Gamma_\HH$ contains a directed cycle $\sigma$.
To see this, let $\gamma$ be a maximal directed path in $\Gamma_\HH$ that
visits each vertex at most once, and let $v$ be the initial vertex of $\gamma$.
If $v \in \BB$, then there is a unique directed edge $e$ in $\Gamma_\HH$ from
some point $w \in \AA$ to $v$, and $e$ is not included in $\gamma$. Likewise,
if $v \in \AA$, then there is an edge $e$ not in $\gamma$ connecting $v$ and
some point $w \in \BB$ since $v$ is not a leaf, and $e$ is directed from $w$ to
$v$ since the only outgoing edge from $v$ is in $\gamma$. In either case, the
maximality of $\gamma$ implies that $w \in \gamma$, which means that $\gamma
\cup e$ contains a directed cycle. However, $(\mu \minus \sigma) \cup (\sigma
\minus \mu)$ is then another perfect matching for $\Gamma_\HH$.

Thus, the Heegaard diagram $\HH$ is a torus with a single $\alpha$ curve and a
single $\beta$ curve intersecting in a single point, which describes the
standard genus-1 Heegaard splitting of $S^3$.
\end{proof}

\bibliographystyle{amsplain}
\bibliography{bibliography}

\providecommand{\bysame}{\leavevmode\hbox to3em{\hrulefill}\thinspace}
\providecommand{\MR}{\relax\ifhmode\unskip\space\fi MR }
\providecommand{\MRhref}[2]{%
  \href{http://www.ams.org/mathscinet-getitem?mr=#1}{#2}
}
\providecommand{\href}[2]{#2}
\begin{thebibliography}{1}

\bibitem{BoyerGordonWatson}
Steven Boyer, Cameron~McA. Gordon, and Liam Watson, \emph{On {L}-spaces and
  left-orderable fundamental groups}, preprint (2011), \arxiv{1107.5016}.

\bibitem{BoyerRolfsenWiest}
Steven Boyer, Dale Rolfsen, and Bert Wiest, \emph{Orderable 3-manifold groups},
  Ann. Inst. Fourier (Grenoble) \textbf{55} (2005), no.~1, 243--288.

\bibitem{GreeneAlternating}
Joshua Greene, \emph{Alternating links and left-orderability}, preprint (2011),
  \arxiv{1107.5232}.

\bibitem{GreeneSpanning}
\bysame, \emph{A spanning tree model for the {H}eegaard {F}loer homology of a
  branched double-cover}, preprint (2008), \arxiv{0805.1381}.

\bibitem{LiscaStipsiczInvariants3}
Paolo Lisca and Andr{\'a}s~I. Stipsicz, \emph{Ozsv\'ath-{S}zab\'o invariants
  and tight contact 3-manifolds. {III}}, J. Symplectic Geom. \textbf{5} (2007),
  no.~4, 357--384.

\bibitem{OSzGenus}
Peter Ozsv{\'a}th and Zolt{\'a}n Szab{\'o}, \emph{Holomorphic disks and genus
  bounds}, Geom. Topol. \textbf{8} (2004), 311--334 (electronic).

\bibitem{OSzProperties}
\bysame, \emph{Holomorphic disks and three-manifold invariants: properties and
  applications}, Ann. of Math. (2) \textbf{159} (2004), no.~3, 1159--1245.

\bibitem{OSz3Manifold}
\bysame, \emph{Holomorphic disks and topological invariants for closed
  three-manifolds}, Ann. of Math. (2) \textbf{159} (2004), no.~3, 1027--1158.

\bibitem{PetersLSpaces}
Thomas Peters, \emph{On {L}-spaces and non left-orderable 3-manifold groups},
  preprint (2009), \arxiv{0903.4495}.

\end{thebibliography}

\end{document}